\documentclass[11pt]{amsart}
\usepackage{amsmath,amsthm,amssymb,graphicx,mathtools,tikz,hyperref}
\usepackage{ dsfont }
\usetikzlibrary{positioning}
\newcommand{\n}{\mathbb{N}}
\newcommand{\z}{\mathbb{Z}}

\newcommand{\real}{\mathbb{R}}

\newcommand{\p}{\mathbb{P}}
\newcommand{\e}{\mathbb{E}}

\DeclareMathOperator*{\argmax}{arg\,max}

\usepackage{mdframed}
\usepackage{lipsum}
\newmdtheoremenv{defi}{Definition}
\newmdtheoremenv{theo}{Theorem}
\newmdtheoremenv{prop}{Proposition}
\newmdtheoremenv{lemm}{Lemma}
\newmdtheoremenv{coro}{Corollary}
\newmdtheoremenv{conj}{Conjecture}
\newmdtheoremenv{claim}{Claim}
\newmdtheoremenv{obs}{Observation}
\newmdtheoremenv{rmk}{Remark}
\newmdtheoremenv{qst}{Question}

\theoremstyle{definition}

\hypersetup{
	colorlinks,
	linkcolor=blue
}
\usepackage{setspace}
\usepackage{hyperref}
\hypersetup{colorlinks=true, allcolors=black}
\AtBeginDocument{\hypersetup{colorlinks=true, allcolors=black}}

\title{$L_1$ and $L_2$ embeddings of the symmetric group}
\author{Cosmas Kravaris}
\begin{document}
\begin{center}
    \maketitle
    \begin{abstract}
    We show that the Cayley graph of the symmetric group $Sym_n$ generated by the cycle $(123...n)$ and the transposition $(12)$ embeds into $L_1$ with bi-Lipschitz distortion $O(1)$.
    This answers a question of Ostrovskii, and along with Kassabov's theorem gives the first example of a sequence of groups which embed bi-Lipschitzly into $L_1$ for one choice of bounded size generating sets, but not for another choice of bounded size generating sets.
    In particular, the Cayley graphs generated by the cycle and the transposition cannot contain coarsely any unbounded sequence of expander graphs.
    Moreover, within the context of the Ribe program, they are a new example of bounded degree Cayley graphs which are test spaces for Rademacher type.
    \end{abstract}
\end{center}

\section{\textbf{Introduction}}

The \textbf{rank} of a finitely generated group $G$, denoted by $rank(G)$, is the smallest size of a generating set of $G$.
Given a group $G$ with a finite generating set $S$, denote by $c_1(G,d_S)$ the bi-Lipschitz distortion of the shortest path metric $d_S$ of the Cayley graph $\Gamma(G,S)$ into $L_1$ (see the background below).

We start with the following observation and fundamental question, shown to us by Naor.
A sequence of finite groups $\{G_n\}_n$ of bounded rank, $\sup_n rank(G_n) < \infty$, tautologically exhibits one of the following three behaviors:
\\\textbf{Never-$L_1$ behavior}: For any sequence of generating sets $\{S_n \subset G_n\}_n$ 
with $\sup_n |S_n| < \infty$,
$$\sup_{n} c_1(G_n,d_{S_n}) = \infty.$$
\textbf{Always-$L_1$ behavior}: For any sequence of generating sets $\{S_n \subset G_n\}_n$ with $\sup_n |S_n| < \infty$,
$$\sup_{n} c_1(G_n,d_{S_n}) < \infty.$$
\textbf{Mixed-$L_1$ behavior}: There exist two sequences of generating sets $\{S_n \subset G_n\}_n$ and $\{S'_n \subset G_n\}_n$ with $\sup_n |S_n| < \infty$ and $\sup_n |S'_n| < \infty$ such that
$$\sup_{n} c_1(G_n,d_{S_n}) < \infty\;\;\;and\;\;\;\sup_{n} c_1(G_n,d_{S'_n}) = \infty.$$
The existence of a sequence of groups with Never-$L_1$ behavior
follows from a theorem of Breuillard and Gamburd \cite{breuillard2010strong},
and the well-known fact that expander graphs do not embed bi-Lipschitzly into $L_1$ \cite{linial1995geometry}.
It was shown to us by Naor that bounded rank abelian groups exhibit Always-$L_1$ behavior (see Proposition \ref{AbelImpliesEuclid} in Section \ref{EuclideanSection}).
See Section \ref{EuclideanSection} for the $L_2$-counterparts of these behaviors.
\\\textbf{The main result of this paper, see Theorem \ref{MainTheorem} below, provides the first example of a sequence of groups with Mixed-$L_1$ behavior.}
\begin{qst}[Naor]\label{QuestionCharacterizeTrichotomy}
  Classify (sequences of) groups $\{G_n\}$ according to the above trichotomy.
  Namely, obtain a useful structural characterization of when they exhibit never, always or mixed $L_1$ behavior.
\end{qst}

In particular, Naor highlighted the following concrete particular case of Question \ref{QuestionCharacterizeTrichotomy}, which could be quite tractable and enlightening to resolve: into which part of the above trichotomy do the lamplighter groups $\{\z/2\z \wr \z/n\z\}_n$ fall?

A landmark result of Kassabov \cite{kassabov2007symmetric} states that there exists a sequence of generating sets $T_n \subset Sym_n$ of the symmetric groups of bounded size, $\sup_n |T_n| < \infty$, such that the Cayley graphs $\{\Gamma(Sym_n, T_n)\}_n$ form a family of expanders.
In light of Kassabov's theorem, Ostrovskii asked the following question \cite{ostr2025personal} which was subsequently posted in \cite{aim_metric_embeddings_participants}.
\begin{qst}[Ostrovskii]\label{OstrovskiiQuestion}
  Do there exist $C \in (0,\infty)$ and $r \in \n$ and generating sets $T_n \subset Sym_n$ each of size $|T_n|\leq r$ such that for each $n$, the symmetric group with word metric $(Sym_n, d_{T_n})$ embeds into $L_1$ with bi-Lipschitz distortion $\leq C$?  
\end{qst}

We give a positive answer to Question \ref{OstrovskiiQuestion}.
Combined with Kassabov's theorem (and the well-known fact that bounded degree expander graphs do not embed bi-Lipschitzly into $L_1$ \cite{linial1995geometry}),
we conclude that the symmetric groups $\{Sym_n\}_{n=1}^\infty$ exhibit Mixed-$L_1$ behavior.

\begin{theo}[Cycle and transposition live in $L_1$]\label{MainTheorem}.\\
    For any $n \in \n$, the symmetric group $Sym_n$ on $\z/n$ endowed with the word metric generated by the transposition $t:= (01)$ and the cycle $c:= (0123...n-1)$ 
    \\embeds into $L_1$ with bi-Lipschitz distortion $O(1)$. Specifically, we have
    $$\sup_{n} c_1(Sym_n,d_{\{t,c\}}) < 1000.$$
\end{theo}
The constant $1000$ is not optimized.

It is well-known that $\Gamma(Sym_n, \{t,c\})$ is not an expander (e.g. see Remark 2 on page 510 in \cite{BABAI1989507} or Section 11.4 in the survey \cite{hoory2006expander}).
Theorem \ref{MainTheorem} implies that $\Gamma(Sym_n, \{t,c\})$ is \textbf{very far} from being an expander in the following strong sense.

\begin{coro}[Coarse opposite to Kassabov's theorem]\label{CoarseAntiKassabov}
    .\\The sequence of Cayley graphs $\{\Gamma(Sym_n,\{t,c\})\}_n$ cannot contain coarsely an unbounded sequence of bounded degree expanders.
    That is, given two increasing functions $\alpha, \beta: [0,\infty) \to [0,\infty)$ with $\lim_{t \to \infty} \alpha(t) = \infty$,
    and a finite graph $\Gamma$ with (normalized) spectral gap $1-\lambda$ and maximum degree $\Delta$,
    \textbf{if} there exists $f: \Gamma \to Sym_n$ with
    $$\alpha(d_\Gamma(u,v)) \leq d_{\{c,t\}}(f(u),f(v)) \leq \beta(d_{\Gamma}(u,v))\;\;\;for\;all\;u,v \in V(\Gamma)$$
    (where $d_{\Gamma}$ denotes the shortest path metric on $\Gamma$ and $d_{\{c,t\}}$ the word metric on $Sym_n$),
    \\\textbf{then} the number of vertices in $\Gamma$ is bounded:
    $$|V(\Gamma)| \leq \Delta^{\alpha^{-1}(C\beta(1)/(1-\lambda))}$$
    where $C>0$ is a universal constant.
\end{coro}

The proof that Theorem \ref{MainTheorem} implies Corollary \ref{CoarseAntiKassabov} follows from Gromov \cite{gromov2003random} and it is a standard argument nowadays (see Theorem 5.7 and the proof of Theorem 4.9 in \cite{ostrovskii2013metric}).
Note that the dependence on the max-degree $\Delta$ 
is necessary: $\{\Gamma(Sym_n,\{t,c\})\}_n$ contain the Hamming cubes bi-Lipschitzly which are unbounded degree spectral expanders (see Subsection \ref{SubSectionRibe} and Section \ref{SectionProofOfCorollaries}).

\subsection{Background}

The \textbf{$L_1$ bi-Lipschitz distortion} of a metric space $(\mathcal{M},d)$, denoted by $c_1(\mathcal{M})$, is the infimal $D>1$ for which there exists a map $f: \mathcal{M} \to L_1$ with
$$d(x,y) \leq ||f(x)-f(y)||_1\leq  D\;d(x,y)\;\;\;for\;all\;x,y \in \mathcal{M}.$$
Understanding the smallest possible distortion $D>1$ for embedding a metric space is a central question in the theory of metric embeddings (see for instance \cite{matousek2013lectures, ostrovskii2013metric, naor2012introduction, deza1997geometry, hoory2006expander}).
The question is equivalent to bi-Lipschitz approximation by a measured walls structure which has been fruitfully investigated in the context of geometric group theory 
(see for example Prop. 2.6 in \cite{cornulier2012proper}, \cite{chatterji2025median}, Chapter 6 in \cite{dructu2018geometric} and references within).

For a Banach space $X$, the bi-Lipschitz distortion $c_X(\mathcal{M})$ is defined analogously. When $X=L_2$ we write $c_2(\mathcal{M}):=c_{L_2}(\mathcal{M})$.
$L_2$-distortion will be discussed in Section \ref{EuclideanSection}.
(We remark that for finite metric spaces, whether the target space is $L_1([0,1])$ or $l_1(\n)$ makes no difference. For every $\epsilon>0$, any finite dimensional subspace of one $L_1$-space is $(1+\epsilon)$-isomorphic to some finite dimensional subspace of the other $L_1$-space. The same remark holds for $L_2$.)

\textbf{Notation:} 
For each $n \in \n$ let $Sym_n$ be the symmetric group on $\z/n = \{0,1,2,...,n-1\}$, $t = (01)$ be the transposition of the first two elements and $c = (0123...n-1)$ be the cyclic permutation.
We consider the (left) Cayley graphs of the symmetric groups $\Gamma(Sym_n,\{t,c\})$ with vertices $Sym_n$ and edges $\{(t \pi,\pi), (c\pi,\pi): \pi \in Sym_n\}$ and view $Sym_n$ as a metric space endowed with the shortest path metric (that is, the word metric).
We denote the metric by $d(\cdot, \cdot)$ and the distance to the identity (or word length) by $|\cdot|$.
We will write: $S := \{t,c\}$ for the generating set.
For each $k,l \in \z/n$ we denote by $d_{\z/n}(k,l)$ the distance between $k$ and $l$ on the Cayley graph of $\z/n$ generated by $\{+ 1\}$.
Finally, the product of two permutations is read from right to left, as in function composition.
This means that, for instance, $(12)(23) = (123)$ and $(23)(12) = (132)$.

We use the following asymptotic notation: two sequences $\{a_n\}_n, \{b_n\}_n \subset \real^+$ one has $a_n \lesssim b_n$ if and only if $b_n \gtrsim a_n$ if and only if there exists $0<C<\infty$ such that $a_n \leq C b_n$ for all $n \in \n$. Also, we write $a_n \asymp b_n$ when $a_n \lesssim b_n$ and $b_n \lesssim a_n$ .

\textbf{Remark:} It is well-known that if we do not insist that the generating sets are of bounded size, then it is easy to embed the symmetric groups into $L_1$.
For each $n$, take the generating set of all transpositions.
Map $f: Sym_n \to l_1([n]^2)$ by mapping each permutation $\pi$ to its associated permutation matrix $A_\pi$.
Whenever two permutations differ by a transposition, their difference in the image is exactly $2$.
By the triangle inequality, this gives the upper bound $||A_\pi-A_\tau||_1 \leq 2 d(\pi,\tau)$ for any $\pi, \tau \in Sym_n$.
For the lower bound, observe that $||A_\pi - A_\tau||_1 = ||Id - A_{\tau \pi^{-1}}||_1$ so it suffices to show $||Id - A_\pi|| \gtrsim d(1,\pi)$.
Observe that by our choice of generating set, $d(1,\pi) = \#\{k \in [n]: \pi(k) \neq k\} - \#\{cycles\; of \;\pi \}$ and also $||Id - A_\pi||_1 = 2  \#\{k \in [n]: \pi(k) \neq k\}$; so the lower bound also follows.

\subsection{New example of bounded degree Cayley test spaces for Rademacher type}\label{SubSectionRibe}
.\\In this subsection and in Section \ref{SectionProofOfCorollaries} we discuss an application of Theorem \ref{MainTheorem} to the geometry of Banach spaces.
Those who are not interested in Banach spaces can skip these parts.
.\\A central aspect of the Ribe program \cite{naor2012introduction, ball2012ribe} aims to provide metric characterizations of local properties of Banach spaces 
(i.e. properties that depend only on the linear structure of the finite dimensional subspaces).
An important early result is a metric characterization of Rademacher type \footnote{Rademacher type is a fundamental invariant of the local geometry of a Banach spcae (see \cite{maurey2003type}). We do not need to recall it here in this paper, because we will not use it.} through test spaces:
a Banach space $X$ has trivial Rademacher type if and only if $X$ contains the Hamming cubes with uniformly bounded bi-Lipschitz distortion, i.e. 
$\sup_n c_X((\{0,1\}^n,||\cdot||_1)) < \infty$.
The "only if" direction follows from Pisier \cite{pisier1973espaces}.
For the "if" direction note the folklore result that for every $0<\epsilon<1$ and every finite subset $F \subset L_1$, $F$ embeds into $\{0,1\}^n$ with distortion $1+\epsilon$ for some $n\in \n$ (this follows from a standard approximation argument; see for instance the proof of Lemma 5.4 in \cite{mendel2014expanders}).
By Ribe's theorem \cite{ribe1976uniformly} (see also \cite{bourgain2006remarks} for quantitative bounds), taking an $\epsilon$-net of the unit ball of any finite dimensional subspace of $L_1$ we see that $L_1$ is crudely finitely representable in $X$, hence $X$ has trivial Rademacher type.
(Almost sharp quantitative bounds for this metric characterization of type follow from \cite{bourgain1986type} and truly sharp bounds follow from \cite{ivanisvili2020rademacher}.)

A sequence of metric spaces $((M_n,d_n))_n$ is a \textbf{family of test spaces} for a Banach space $Z$ whenever for all Banach spaces $X$ we have that $Z$ is finitely representable in $X$ if and only if $\sup_n c_X((M_n,d_n)) < \infty$ \cite{ostrovskii2013metric}.
Ostrovskii \cite{ostrovskii2016metric, ostrovskii2011different} showed that for any Banach space $Z$ there exists a family of $3$-regular graphs whose shortest-path metrics form a family of test spaces for $Z$.
However, these test spaces are not vertex-transitive, even for $L_1$.

A theorem of Naor and Peres \cite{naor2008embeddings} states that the lamplighter groups $\z/n \wr \z/n$ with the standard generators (either move the pointer by a unit or edit the current lamp by a unit) embed into $L_1$ with uniformly bounded bi-Lipschitz distortion.
On the other hand, Arzhantseva, Guba and Sapir \cite{arzhantseva2006metrics} showed that every Hamming cube $\{0,1\}^n$ embeds into $\z/m \wr \z/m$ for some $m$ with uniformly bounded bi-Lipschitz distortion.
It follows that the sequence of $4$-regular Cayley graphs of $(\z/n \wr \z/n)_n$ form test spaces for Rademacher type.
As an application of Theorem \ref{MainTheorem}, we have a new example of bounded degree Cayley test spaces of Rademacher type which are $3$-regular and are Cayley graphs of the symmetric groups.

\begin{prop}[Bounded degree Cayley test spaces for Rademacher type].\\
    A Banach space $X$ has trivial Rademacher type if and only if 
    $$\sup_n c_X(Sym_n,d_{\{t,c\}}) < \infty.$$
\end{prop}

\subsection{$L_2$-embeddings of Cayley graphs into Hilbert space}
.\\A conjecture of Cornulier-Tessera-Valette \cite{de2008isometric} states that a finitely generated group whose word metric embeds bi-Lipschitzly into Hilbert space must have an abelian subgroup of finite index.
In Section \ref{EuclideanSection}, we state a version of this conjecture due to Naor about finite groups, and prove that a positive answer to the conjecture of Cornulier-Tessera-Valette implies a positive answer to the conjecture of Naor.
One should view Section \ref{EuclideanSection} as the Euclidean counterpart to Question \ref{QuestionCharacterizeTrichotomy}.
This section also contains several results and conjectures due to Naor which appear for the first time with his permission.

\subsection{Overview of the proof of Theorem 1 and paper organization}
.\\In Section \ref{SectionYadinFormula} we begin with an explicit formula for the word metric of $(Sym_n,\{c,t\})$ up to constant factors.
The formula is inspired by the word metric formula of Yadin for the lampshuffler groups \cite{yadin2009rate}.
The \textit{lampshuffler group} $Sym_{00}(\z) \ltimes \z$ consists of all pairs $(\pi,x)$ where $\pi$ is a finite support permutation on $\z$ and $x \in \z$ is a pointer.
Roughly speaking, this is a variation of the lamplighter on $\z$, but the generator moves are: just moving the pointer and moving the pointer while transposing the two elements along the movement.
In fact, it is not hard to show that the lampshuffler group $Sym_{00}(\z) \ltimes \z$ embeds bi-Lipschitzly into $L_1$ (see \cite{kravaris2026forthcoming}).
In the symmetric group $Sym_n$, the cycle $c$ plays the role of the "pointer" and $t$ takes the role of the transposition generator
(though no knowledge of lamplighter nor lampshufflers is required for the proof).

The \textbf{key difficulty} in analyzing the word metric on $Sym_n$ as opposed to $Sym_{00}(\z) \ltimes \z$ (or even $Sym_{00}(\z/n) \ltimes \z/n$) is that \underline{\textbf{the pointer (or frame of reference) is ambiguous}}.
For example, whenever $\pi = c^2 t c t c^{-4}$ then it is obvious that the "pointer" should be $-1$, but for a more complicated permutation this is not clear.
The first step of the proof is to show that the distance between two permutations $\pi$ and $\tau$ is given by:
$$|\tau \pi^{-1}| \asymp \min_{l \in \z/n}\left(\sum_{k \in \z/n} d_{\z/n}(\pi(k) - l,\tau(k)) + diam_{\z/n}(\{0,l\}\cup\{p:\pi^{-1}(p) \neq \tau^{-1}(p-l)\})\right)$$
where $d_{\z/n}$ denotes the distance on the $n$-cycle $\z/n$.
Observe that we have to take the minimum over all potential "pointer positions" $l \in \n$.

\textbf{The presence of this minimum makes the metric difficult to analyze}.

In Section \ref{SectionSplitingInfimum} we split the minimum of the sum of two terms into the sum of two minimums:
$$|\tau \pi^{-1}| \asymp \min_{l \in \z/n}\sum_{k \in \z/n} d_{\z/n}(\pi(k) - l,\tau(k)) + \min_{l \in \z/n} diam_{\z/n}(\{0,l\}\cup\{p:\pi^{-1}(p) \neq \tau^{-1}(p-l)\})=: T_1 + T_2$$
This is achieved by high-low casework.
Whenever $T_1 < n/3$, for some $l$ there are at most $n/3$ nonzero terms in the sum, so the "potential pointer" is obvious, and the same value $l$ will be the optimal one in all three minimums.
Whenever $T_1 > n/3$, the first term dominates.
This is because we always have $T_2 \leq n/2$ (since it is a diameter) and hence: $T_1 \leq T_2 \leq (1 + 3/2)T_2$.
The rest of the proof deals with embedding each term separately.

The first term, 
$\min_{l \in \z/n}\sum_{k \in \z/n} d_{\z/n}(\pi(k) - l,\tau(k))$,
is a metric on $Sym_n$ and can be viewed as a subset of the abelian group $[n-1]^{\z/n}$ (where addition is pointwise addition of functions $f: [n-1] \to \z/n$) with word metric given by the generating set $\{(1,0,...,0), (0,1,0,...,),..., (0,...,0,1), (1,1,...,1)\}$.
We follow the embedding method of Austin-Naor-Valette \cite{austin2010euclidean} (see also section 4 in Naor-Peres \cite{naor2008embeddings}),
and write the embedding into $L_1$ using the representation theory of abelian groups (that is, exponential sums).
We interpret 
$$\min_{l \in \z/n}\sum_{k \in \z/n} d_{\z/n}(\pi(k) - l,\tau(k)) = \min_{l \in \z/n}\sum_{k \in \z/n} d_{\z/n}(\pi(k) - \tau(k),l)$$
as the minimum sum of distances of the cloud of points $\{\pi(k)-\tau(k)\}$ to a "median" point $l \in \z/n$.
It is well-known (e.g. see \cite{rabinovich2008average}) that
this quantity is, up to a factor of $2$, the average distance of the cloud of points.
Based on this observation, we construct the embedding via exponential sums.
(Sidenote: similar to how \cite{austin2010euclidean} and \cite{naor2008embeddings} proceed, the embedding was found by reverse-engineering: first searching through all possible $L_2$-representations of this abelian group and pointing out which representations will not work.)

The second term, 
$\min_{l \in \z/n} diam_{\z/n}(\{0,l\}\cup\{p:\pi^{-1}(p) \neq \tau^{-1}(p-l)\})$
is a variation of the lamplighter metric, but in which the lamplighter has "forgotten his position".
It was shown by Naor-Peres \cite{naor2008embeddings} that the lamplighter metric embeds into $L_1$. (Later Ostrovskii-Randrianantoanina \cite{ostrovskii2019characterization} gave a different embedding which works more generally into any non-superreflexive Banach space.)
Here we modify the second embedding given in Naor-Peres by identifying/collapsing certain specific coordinates of $l_1(\n)$.
Whenever the first term is small, e.g. $T_1<n/3$, the optimal "pointer" $l$ is obvious, and the coordinate identifications do not "influence" the lower bound of the embedding.

In Section \ref{SectionFinishingTheProof} we combine the estimates for each of the two terms and prove Theorem \ref{MainTheorem}.

In Section \ref{EuclideanSection} we discuss $L_2$-embeddings of Cayley graphs into Hilbert space.

Finally, in Section \ref{SectionProofOfCorollaries} we prove the test space characterization for $L_1$.
The only missing step is showing that the Hamming cube embeds into the Cayley graph of a cycle and a transposition.
The embedding is similar to that of Arzhantseva, Guba and Sapir \cite{arzhantseva2006metrics} for $\z/n \wr \z/n$,
and the analysis uses the word metric formula in Section \ref{SectionYadinFormula}.

\section{\textbf{The word metric of cycle and transposition}}\label{SectionYadinFormula}

We begin by slowly examining the word length of various types of permutations.
    \\\textbf{Adjacent transpositions:} Via conjugating by cyclic permutations, we can obtain any transposition between adjacent numbers:
    $$c^{k} t c^{-k} = (k(k+1))\;\;\;for\;all\;k \in \z/n.$$
    \textbf{General transpositions:} Any transposition between non-adjacent numbers can be obtained by applying transpositions between adjacent numbers.
    For example, $(13)=(12)(23)(12)$.
    More generally, for any $k \in \z/n$, $(k(k+2)) = (k(k+1))((k+1)(k+2))(k(k+1)).$
    Substituting the expression of $(k(k+1))$ in terms of our generators, we observe \textit{cancellation} between the conjugation exponents:
    $$(k(k+2)) = c^{k} t c^{-k} c^{k+1} t c^{-k-1} c^{k} t c^{-k} = c^{k} t c t c^{-1} t c^{-k}.$$
    We now want to express a general transposition as the product of generators. For every $2 \leq l \leq \lfloor n/2\rfloor + 1$ we have:
    $$(0l)=(01)(12)...((l-2)(l-1))((l-1)l)((l-2)(l-1))...(12)(01) = (tc)^{l-1} t (tc)^{-(l-1)}$$
    while for $\lfloor n/2\rfloor + 1 \leq l \leq n-1$ we have a shorter expression since $(0(n-1)) = c^{-1} t c =: t'$
    $$(0l) = (t'c^{-1})^{l-1} t' (t'c^{-1})^{-(l-1)}.$$
    To obtain a general transposition $(k(k+l))$ we simply conjugate the transposition $(0l)$.
    We have the word length estimates:
    $$|(0l)|_S \leq 4 d_{\z/n}(0,l)\;\;\;and\;\;\;|(k(k+l))|_S \leq 4 d_{\z/n}(0,l) + 2 d_{\z/n}(0,k).$$
    \textbf{Cyclic permutations:} 
    Take a cyclic permutation $(k_1 k_2 ... k_m)$ where $m \in [n]$ and $k_1,...,k_m \in \z/n$ are distinct.
    We write the cyclic permutation as the product of transpositions and observe that there is \textit{cancellation} between the conjugation exponents:
    $$(k_1...k_m) = (k_1 k_2) (k_2 k_3) ... (k_{m-2} k_{m-1}) (k_{m-1} k_m)$$
    $$= c^{k_1} (0 (k_2 - k_1))c^{-k_1} c^{k_2} (0 (k_3 - k_2))c^{-k_2} ... c^{k_{m-1}} (0 (k_m - k_{m-1}))c^{-k_{m-1}}
    = c^{k_1} \left(\prod_{i=1}^{m-1} (0 (k_{i+1} - k_i))c^{k_{i+1}-k_i}\right) c^{-k_m}.$$
    By the triangle inequality we arrive at the estimate:
    $$|(k_1 ... k_m)|_S \leq d_{\z/n}(0,k_1) + d_{\z/n}(0,k_m) + 5 \sum_{i=1}^{m-1} d_{\z/n}(k_i,k_{i+1}) \leq 2d_{\z/n}(0,k_1) + 6\sum_{i=1}^{m-1} d_{\z/n}(k_i,k_{i+1})\;\;\;\;(\star)$$
The following estimate on the word length of a permutation 
is inspired by Propositions 1.3 and 1.4 of Yadin in \cite{yadin2009rate} about lampshuffler groups.

\begin{lemm}
    For each permutation $\pi \in Sym_n$, the word length is given by
    $$|\pi|_{S} \asymp \min_{l \in \z/n}\left(\sum_{k \in \z/n} d_{\z/n}(k,\pi(k)+l) + diam_{\z/n}(\{0,l\}\cup\{p:\pi(p) \neq p-l\})\right).$$
    To be more precise, we have:
    $$\dfrac{1}{3} \min_{l \in \z/n}\left(\sum_{k \in \z/n} d_{\z/n}(k,\pi(k)+l) + diam_{\z/n}(\{0,l\}\cup\{p:\pi(p) \neq p-l\})\right)$$
    $$\leq |\pi|_{S} \leq \min_{l \in \z/n}\left(6\sum_{k \in \z/n} d_{\z/n}(k,\pi(k)+l) + 2 diam_{\z/n}(\{0,l\}\cup\{p:\pi(p) \neq p-l\})\right).$$
\end{lemm}
\begin{proof}
    \textbf{Upper bound (via construction)}
    We will show the upper bound for $l=0$.
    For any $l \in \z/n$, the bound follows by applying the formula to $c^l \pi$.
    Write $\pi$ as the product of disjoint cyclic permutations, say 
    $$\pi = \prod_{j=1}^s c^{l_j} \gamma_j c^{-l_j}$$
    where for each $j=1,...,s$, $l_j \in \z/n$ and $\gamma_j = (0 k^{(j)}_2 k^{(j)}_3...k^{(j)}_{m_j})$ for some $m_j \in \n$ and $k^{(j)}_2,...,k^{(j)}_{m_j} \in \z/n$ all distinct and different from $0$.
    After relabeling the indices, we may assume that the path $0 \mapsto l_1 \mapsto l_2 \mapsto ... \mapsto l_s$ is a path of length at most twice the diameter:
    $$d_{\z/n}(0,l_1) + \sum_{j=2}^{s} d_{\z/n}(l_{j-1},l_j) \leq 2 diam_{\z/n}(\{0\}\cup\{k:\pi(k+l) \neq k\}).$$
    (Why twice? It could be the case that $0$ sits at the middle of the cloud of points $\{l_1,...,l_s\}$.)
    \\We now write:
    $$\pi = \prod_{j=1}^{s} c^{l_j} \gamma_j c^{-l_j} = c^{l_1} \gamma_1 \prod_{j=2}^{s} c^{-l_{j-1}+l_{j}} \gamma_j$$
    and the upper bound follows from the triangle inequality and the estimate on the word length of a cycle (i.e. the inequality $(\star)$).
    \\\textbf{Lower bound (via submultiplicative quantity)}
    We show that the term
    $$T := \min_{l \in \z/n }\left(d_{\z/n}(0,l)+\sum_{k \in \z/n} d_{\z/n}(k,\pi(k+l)) + diam_{\z/n}(\{0\}\cup\{k:\pi(k+l) \neq k\})\right)$$
    does not change much when we apply a generator.
    \\When we multiply $\pi$ by the cyclic permutation, the $T$ changes by at most 1 unit.
    \\When we multiply $\pi$ by a transposition, the $T$ can change by at most 3 units.
    \\For the identity permutation, $T =0$. Moving along the shortest path from $\pi$ to $1$ in the Cayley graph $\Gamma(Sym_n,\{t,c\})$, the $T$ drops to zero, and hence the number of steps to reach the identity is at least
    $$|\pi|_S \geq \dfrac{T}{3} = \dfrac{1}{3} \min_{l \in \z/n }\left(d_{\z/n}(0,l)+\sum_{k \in \z/n} d_{\z/n}(k,\pi(k+l)) + diam_{\z/n}(\{0\}\cup\{k:\pi(k+l) \neq k\})\right),$$
    which gives us the lower bound.
\end{proof}

Applying the formula to the permutation $\tau \pi^{-1}$ we get:
$$|\tau \pi^{-1}| \asymp \min_{l \in \z/n}\left(\sum_{k \in \z/n} d_{\z/n}(k,\tau\pi^{-1}(k)+l) + diam_{\z/n}(\{0,l\}\cup\{p:\tau\pi^{-1}(p) \neq p-l\})\right)$$
$$ = \min_{l \in \z/n}\left(\sum_{k \in \z/n} d_{\z/n}(\pi(k) - l,\tau(k)) + diam_{\z/n}(\{0,l\}\cup\{p:\pi^{-1}(p) \neq \tau^{-1}(p-l)\})\right)$$

\section{\textbf{Splitting the minimum into two terms}}\label{SectionSplitingInfimum}

\begin{lemm}
    For any $\pi, \tau \in Sym_n$, we have
    $$\min_{l \in \z/n}\left(\sum_{k \in \z/n} d_{\z/n}(\pi(k) - l,\tau(k)) + diam_{\z/n}(\{0,l\}\cup\{p:\pi^{-1}(p) \neq \tau^{-1}(p-l)\})\right)$$
    $$\leq 2 \min_{l \in \z/n}\sum_{k \in \z/n} d_{\z/n}(\pi(k) - l,\tau(k)) + \min_{l \in \z/n}diam_{\z/n}(\{0,l\}\cup\{p:\pi^{-1}(p) \neq \tau^{-1}(p-l)\})$$
\end{lemm}
\begin{proof}
    Note that the direction $\geq$ is trivial (when we drop the constant 2).
    The point of this lemma is to show the reverse non-trivial direction.
    We will write:
    $$T_1:= \min_{l \in \z/n}\sum_{k \in \z/n} d_{\z/n}(\pi(k) - l,\tau(k)),\;\;\;
    T_2 := \min_{l \in \z/n}diam_{\z/n}(\{0,l\}\cup\{p:\pi^{-1}(p) \neq \tau^{-1}(p-l)\})$$
    \\\textbf{CASE A:} $T_1 = \min_{l \in \z/n}\sum_{k \in \z/n} d_{\z/n}(\pi(k) - l,\tau(k)) \geq n/2$
    \\Observe that always $T_2 \leq n/2$, so the second term is of lower order compared to the first term, and we get the trivial estimate:
    $$\min_{l \in \z/n}\left(\sum_{k \in \z/n} d_{\z/n}(\pi(k) - l,\tau(k)) + diam_{\z/n}(\{0,l\}\cup\{p:\pi^{-1}(p) \neq \tau^{-1}(p-l)\})\right)$$
    $$\leq T_1 + n/2 \leq T_1 + T_1 \leq 2 T_1 + T_2.$$
    \\\textbf{CASE B:} $T_1 = \min_{l \in \z/n}\sum_{k \in \z/n} d_{\z/n}(\pi(k) - l,\tau(k)) < n/2$
    \\Note that in this case $|\{k: \pi(k)-l \neq \tau(k)\}| < n/2$ for any value of $l$ which attains this minimum.
    Also, via a change of variables:
    $$|\{k: \pi(k)-l \neq \tau(k)\}| = |\{k: k-l \neq \tau(\pi^{-1}(k))\}|
    = |\{k: \tau^{-1}(k-l) \neq \pi^{-1}(k)\}|,$$
    so for any other value $l' \neq l$ we get that
    $$diam_{\z/n}(\{0,l'\}\cup\{p:\pi^{-1}(p) \neq \tau^{-1}(p-l')\})
    \geq |\{k: \tau^{-1}(k-l') \neq \pi^{-1}(k)\}|$$
    $$\geq |\{k: \tau^{-1}(k-l) = \pi^{-1}(k)\}| > n/2.$$
    We have two further subcases:
    \\\textbf{CASE B1}: The minimum of $T_2$ is attained at the same value $l$ as the minimum of $T_1$.
    \\In this case the claim is trivial (since the left-hand side equals $T_1 + T_2$).
    \\\textbf{CASE B2:} The minimum of $T_2$ is attained at a different value $l'\neq l$ than the minimum of $T_1$.
    \\In this case, as we saw above, $T_2 > n/2$ and  we get:
    $$T_1 + T_2 > T_1 + \dfrac{n}{2} = \min_{l \in \z/n}\sum_{k \in \z/n} d_{\z/n}(\pi(k) - l,\tau(k)) + \dfrac{n}{2}$$
    $$ \geq \min_{l \in \z/n}\left(\sum_{k \in \z/n} d_{\z/n}(\pi(k) - l,\tau(k)) + diam_{\z/n}(\{0,l\}\cup\{p:\pi^{-1}(p) \neq \tau^{-1}(p-l)\})\right).$$
    (Aside: since the resulting inequality is strict, CASE B2 can never happen.)
\end{proof}

\section{\textbf{Embedding the 1st term: distance to the median}}\label{SectionFirstTermEmbedding}

\begin{lemm}[Embedding the 1st term]\label{LemmaEmbeddingFirstTerm}
    .\\There exists a map $\Phi_1: Sym_n \to L_1$ such that for all $\pi,\tau \in Sym_n$ we have
    $$||\Phi_1(\pi)-\Phi_1(\tau)||_1 \asymp \min_{l \in \z/n}\sum_{k \in \z/n} d_{\z/n}(\pi(k) - l,\tau(k)).$$
    In particular, the bi-Lipschitz distortion is $<2\pi \approx 6.28$.
\end{lemm}

The following observation is well-known (e.g. see \cite{rabinovich2008average}).

\begin{obs}[Average distance vs average distance to a median in the cloud]
    .\\For any metric space $(X,d)$ and any set of points $x_1,...,x_n \in X$
    $$\dfrac{1}{2}\dfrac{1}{n^2} \sum_{i=1}^n \sum_{j=1}^n d(x_i,x_j) \leq \min_{1 \leq r \leq n} \dfrac{1}{n} \sum_{j=1}^n d(x_r,x_j) \leq \dfrac{1}{n^2} \sum_{i=1}^n \sum_{j=1}^n d(x_i,x_j).$$
\end{obs}
\begin{proof}
    The second inequality is trivial since the minimum of a list of numbers is always $\leq$ the average.
    For the first inequality, for all $r \in [n]$ we use the triangle inequality:
    $$\dfrac{1}{n^2} \sum_{i=1}^n \sum_{j=1}^n d(x_i,x_j) \leq \dfrac{1}{n^2} \sum_{i=1}^n \sum_{j=1}^n (d(x_r,x_j) + d(x_r,x_j))
    = \dfrac{2}{n} \sum_{j=1}^n d(x_r,x_j)$$
\end{proof}

\begin{obs}[On the circle every cloud contains one of its median points]
    .\\For any $n,m \in \n$ and any set of points $x_1,...,x_m \in \z/n$ we have:
    $$\min_{1\leq r \leq m} \sum_{k=1}^m d_{\z/n}(x_k,x_r) = \min_{l \in \z/n} \sum_{k=1}^m d_{\z/n}(x_k,l).$$
\end{obs}
\begin{proof}
    Let $l$ be a minimizer for $\min_{l \in \z/n} \sum_{k=1}^m d_{\z/n}(x_k,l)$.
    If $l \in \{x_1,...,x_m\}$ then we are done, so suppose not.
    Consider the antipodal point $\bar{l}:=\lfloor n/2 \rfloor +l$ and the two intervals $[l,\bar{l}]$ and $[\bar{l},l]$ whose union is the whole circle $\z/n$.
    Without loss of generality, assume that 
    $$|\{x_1,...,x_m\} \cap [\bar{l},l]| \geq |\{x_1,...,x_m\} \cap [l,\bar{l}]|.$$
    Let $x_r \in \{x_1,...,x_m\} \cap [\bar{l},l]$ be the point in $\{x_1,...,x_m\} \cap [\bar{l},l]$ which is closest to $l$.
    We claim that $x_r$ is also a median:
    $$\sum_{k=1}^m d_{\z/n}(x_k,x_r)  = \sum_{x_k \in [\bar{l},l]}^m d_{\z/n}(x_k,x_r) + \sum_{x_k \in (\bar{l},l]}^m d_{\z/n}(x_k,x_r)$$
    $$ \leq \sum_{x_k \in [\bar{l},l]}^m (d_{\z/n}(x_k,l) - d_{\z/n}(x_r,l))  + \sum_{x_k \in (\bar{l},l]}^m (d_{\z/n}(x_k,l) + d_{\z/n}(x_r,l))$$
    $$ = \sum_{k=1}^m d_{\z/n}(x_k,l) - d_{\z/n}(x_r,l)\left(|\{x_1,...,x_m\} \cap [\bar{l},l]| - |\{x_1,...,x_m\} \cap [l,\bar{l}]|\right) \leq \sum_{k=1}^m d_{\z/n}(x_k,l).$$
\end{proof}

\begin{proof}[Proof for embedding the first term in $L_1$]
    We map $\Phi_1: Sym_n \to l_1^{n^2}$ by
    $$\Phi_1(\pi) := \left(e^{2\pi i (\pi(k)-\pi(r))/n}\right)_{k,r \in [n]}\;\;\;for\;all\;\pi \in Sym_n.$$
    For all $\pi,\tau \in Sym_n$ we have:
    $$||\Phi_1(\pi)-\Phi_1(\tau)||_1 = \sum_k \sum_r |e^{2\pi i (\pi(k)-\pi(r))/n} - e^{2\pi i (\tau(k)-\tau(r))/n}|$$
    $$\asymp \sum_k \sum_r \dfrac{1}{n} d_{\z/n}(\pi(k)-\pi(r),\tau(k)-\tau(r))
    = \dfrac{1}{n} \sum_k \sum_r d_{\z/n}(\pi(k)-\tau(k),\pi(r)-\tau(r))$$
    $$\asymp \min_{r} \sum_k d_{\z/n}(\pi(k)-\tau(k),\pi(r)-\tau(r))
    = \min_{l \in \z/n} \sum_k d_{\z/n}(\pi(k)-\tau(k),l).$$
    Finally, observe that in each of the two $\asymp$ steps we pay factors $\pi$ and $2$ respectively.
\end{proof}

\section{\textbf{Embedding the 2nd term: a lamplighter who forgets his position}}\label{SectionSecondTermEmbedding}

\begin{lemm}[Embedding the 2nd term]\label{LemmaEmbeddingSecondTerm}
    .\\There exists a $4$-Lipschitz function $\Phi_2: Sym_n \to L_1$ such that
    for all $\pi, \tau \in Sym_n$
    \\\textbf{if} $\min_l \sum_{k} d(\pi(k),\tau(k)+l) < n/3$, \textbf{then}
    $$||\Phi_2(\pi)-\Phi_2(\tau)||_1 \geq \dfrac{1}{8} \min_{l \in \z/n}diam_{\z/n}(\{0,l\}\cup\{p:\pi^{-1}(p) \neq \tau^{-1}(p-l)\}).$$
\end{lemm}
Denote by $\mathcal{J}$ the set of all intervals $J = [a,b] = \{a, a+1,...,b\} \subset \z/n$ and by $J^{oo} := [a+2,b-2]$ the double-interior of $J$ (which is empty when $|J|\leq 2$).
Given any function $f: J \to \z/n$ and $J=[a,b] \in \mathcal{J}$ we will write $list(f,J)$ for the function
$$list(f,J): \{0,1,...,b-a\} \to \z/n: k \mapsto f(k + a).$$
Let $\{v_{k,f}\}_{k,f}$ be a standard basis of $l_1^N$ where $1 \leq k \leq n$ and $f:\{0,1,...,k\} \to \z/n$ is a function, and $N \in \n$ is the number of all possible pairs $(k,f)$.
Define:
$$\Phi_2(\pi) := \dfrac{1}{n}\sum_{J \in \mathcal{J}} 1_{\{0 \notin J^{oo}\}} v_{|J|,list(\pi^{-1},J)}.$$
The role of the notation $list(\pi^{-1},J)$ is to record the permutation $\pi^{-1}$ on the interval $J$ but forget the starting point of the interval.
The above embedding is similar to the one in \cite{naor2008embeddings} except for this identification of the coordinates.
\\\textbf{Upper Bound} We show that $\Phi_2$ is Lipschitz with respect to the word metric. Fix $\pi \in Sym_n$.
\\\textbf{Across a transposition edge:} $\pi^{-1}$ and $((01)\pi)^{-1}$ are identical functions except at the points $0$ and $1$. 
This means that for all $J \in \mathcal{J}$ with 
$$\{0,1\}\cap J = \emptyset \iff v_{|J|,list(\pi^{-1},J)} = v_{|J|,list(((01)\pi)^{-1},J)}.$$
Also, if $0 \in J^{oo}$, then the coefficient of $v_{|J|,list(\pi^{-1},J)}$ will vanish and likewise for $v_{|J|,list(((01)\pi)^{-1},J)}$.
There are at most $4n$ intervals $J$ with $\{0,1\}\cap J \neq \emptyset$ and $0 \notin J^{oo}$.
\\We conclude that $||\Phi_2(\pi)-\Phi_2((01)\pi)||_1 \leq 4$.
\\\textbf{Across a cyclic permutation edge:} 
Observe that for all $J \in \mathcal{J}$,
$$list(\pi^{-1},J) = list((c\pi)^{-1},J+1),$$
so the corresponding coordinates are identical.
The only way to get a nonzero summand in the $l_1^N$ norm is whenever $0 \in J^{oo}$ and $0 \notin (J+1)^{oo} = J^{oo} +1$ or whenever $0 \notin J^{oo}$ and $0 \in (J+1)^{oo} = J^{oo} +1$.
The number of such $J$ is $\leq 2n$,
\\We conclude that $||\Phi_2(\pi)-\Phi_2(c\pi)||_1 \leq 2$.

\textbf{Lower Bound}
\\To avoid any ambiguity about the interval notation, in what follows, for all $x,y,z \in \z/n$, we write $x < y < z$ to say that the counterclockwise path from $x$ to $z$ passes through $y$.
Also, we use the notation $J=[a,b]:=\{x \in \z/n: a \leq x \leq b\}$. (That way, for instance, $[0,n-3]$ is an interval of size $n-2$ whereas $[n-3,0]$ is an interval of size $4$.)
\begin{obs}\label{ManyIntervalsFeelMass}
    For any subset $S \subset \z/n$ with $|S| \leq n/3$ we have:
   $$|\{J \in \mathcal{J}|\emptyset \neq J \cap S \neq J\;and\; 0 \notin J^{oo}\}| \geq \dfrac{n}{4} \;diam_{\z/n}(\{0\}\cup S).$$
\end{obs}

\begin{proof}
    For notational simplicity assume that $n$ is even. We consider the points 
    $$x:= \argmax_{x \in S} d_{\z/n}(x,0),\;and\;\;\;y:= \argmax_{y \in \z/n - S} d_{\z/n}(y,0).$$
    \textbf{CASE A}: $d(0,x) \neq n/2 = diam(\z/n)$.
    \\Without loss of generality suppose that $0 < x < n/2$ (so $x$ in the "right" side of the clock).
    By definition $x+1 \notin S$.
    Any interval $J = [a,b] \in \mathcal{J}$ with
    $0 \leq a \leq x < x+1 \leq b \leq 0$
    will satisfy $ \emptyset \neq J \cap S \neq J$ and $0 \notin J^{oo}$.
    In total, the number of such intervals is
    $$\geq (n-  d_{\z/n}(x,0)) d_{\z/n}(x,0) \geq (n/2)  d_{\z/n}(x,0) \geq (n/4) diam_{\z/n}(\{0\}\cup S).$$
    \textbf{CASE B}: $x = n/2$.
    \\Without loss of generality, suppose that $0 < y < n/2$. By definition $y+1 \in S$.
    Also, since $|S| \leq n/3$, we have $d_{\z/n}(y,0) \geq n/3$.
    Any interval $J = [a,b] \in \mathcal{J}$ with
    $0 \leq a \leq y < y+1 \leq b \leq 0$
    will satisfy $ \emptyset \neq J \cap S \neq J$ and $0 \notin J^{oo}$.
    In total, the number of such intervals is
    $$\geq (n-  d_{\z/n}(y,0)) d_{\z/n}(y,0) \geq (n/2)(n/3) = (n/3) diam_{\z/n}(\{0\}\cup S)$$
\end{proof}
Now fix $\pi, \tau \in Sym_n$
with $\min_l \sum_{k} d(\pi(k),\tau(k)+l) < n/3$,
so there exists an "obvious pointer" $l \in \z/n$ such that
$|\{p:\pi^{-1}(p) \neq \tau^{-1}(p-l)\}|<n/3$.
We denote the set of "positions of different entries" by
$$D:= \{p:\pi^{-1}(p) \neq \tau^{-1}(p-l)\}.$$
It is obvious that if an interval $J\in \mathcal{J}$ is disjoint from this set, $J \cap D = \emptyset$, then we get identical coordinates $v_{|J|,list(\pi^{-1},J)} = v_{|J|,list(\tau^{-1},J+l)}$.
We also have a partial converse:

\begin{obs}[some important coordinates have not been collapsed]\label{ObservationCoordinateCollapse}
.\\\textbf{If} $\emptyset \neq J \cap D \neq J$
\textbf{then} $v_{|J|,list(\pi^{-1},J)} \neq v_{|J'|,list(\tau^{-1},J')}$
for all $J' \in \mathcal{J}$.
\end{obs}
\begin{proof}
Pick $p \in J-D$ and observe that $\pi^{-1}(k) = \tau^{-1}(k-l)$ so the only way we could have $v_{|J|,list(\pi^{-1},J)} = v_{|J'|,list(\tau^{-1},J')}$ was if $J' = J+l$. (This is because in a permutation, each element appears only once.)
Picking any other point $p' \in J \cap D$, we see that
$v_{|J|,list(\pi^{-1},J)} \neq v_{|J|,list(\tau^{-1},J+l)}$
showing the claim.   
\end{proof}

We use Observations \ref{ManyIntervalsFeelMass} and \ref{ObservationCoordinateCollapse} to each of the summands in the definition of $\Phi_2(\pi)$ and $\Phi_2(\tau)$.
$$||\Phi_2(\pi)-\Phi_2(\tau)||_1 \geq 
\dfrac{1}{n}|\{J| \emptyset \neq J \cap D \neq J\;and\; 0 \notin J^{oo}\}| 
+ \dfrac{1}{n}|\{J| \emptyset \neq (J+l) \cap D \neq (J+l)\;and\;0 \notin J^{oo}+l\}|$$
$$\geq \dfrac{1}{4} \left(diam(\{0\}\cup D) + diam(\{-l\} \cup D)\right) \geq \dfrac{1}{8}diam(\{0,l\} \cup D)$$
$$=  \dfrac{1}{8} \min_{l \in \z/n}diam_{\z/n}(\{0,l\}\cup\{p:\pi^{-1}(p) \neq \tau^{-1}(p-l)\}),$$
and the proof of Lemma \ref{LemmaEmbeddingSecondTerm} is complete.

\section{\textbf{Putting everything together and finishing the proof of theorem \ref{MainTheorem}:}}\label{SectionFinishingTheProof}

Let $\Phi_1$ and $\Phi_2$ be the two embeddings in the above two sections 
and take their direct sum $\Phi=\Phi_1 \oplus \Phi_2: \pi \mapsto (\Phi_1(\pi),\Phi_2(\pi)) \in L_1 \oplus L_1$. We may rescale the map $\Phi_1$ so that $||\Phi_1||_{Lip} \leq 1$ and $||\Phi_1^{-1}||_{Lip} \leq 2 \pi$ (where by $\Phi_1^{-1}$ we mean the inverse on the image).
\\For all $\pi,\tau \in Sym_n$ we have two cases:
\\\textbf{CASE 1: $\min_l \sum_{k} d(\pi(k),\tau(k)+l) \geq n/3$}.
\\For the lower bound:
$$d(\pi,\tau) \geq \dfrac{1}{3} \min_{l \in \z/n}\left(\sum_{k \in \z/n} d_{\z/n}(\pi(k) - l,\tau(k)) + diam_{\z/n}(\{0,l\}\cup\{p:\pi^{-1}(p) \neq \tau^{-1}(p-l)\})\right)$$
$$\geq \dfrac{1}{3} \min_{l \in \z/n}\sum_{k \in \z/n} d_{\z/n}(\pi(k) - l,\tau(k))
\geq \dfrac{1}{3} ||\Phi_1(\pi) - \Phi_1(\tau)||_1$$
$$\geq \dfrac{1}{3} ||\Phi_1(\pi) - \Phi_1(\tau)||_1 + \dfrac{1}{3}||\Phi_2(\pi) - \Phi_2(\tau)||_1 - \dfrac{4}{3}\; d(\pi,\tau).$$
where we used $||\Phi_2(\pi)-\Phi_2(\tau)||_1 \leq 4 d(\pi,\tau)$. 
After rearranging terms, we obtain:
$$d(\pi,\tau) \geq 7 (||\Phi_1(\pi) - \Phi_1(\tau)||_1 + ||\Phi_2(\pi) - \Phi_2(\tau)||_1).$$
while for the upper bound:
$$d(\pi,\tau) \leq 6 \min_{l \in \z/n}\left(\sum_{k \in \z/n} d_{\z/n}(\pi(k) - l,\tau(k)) + diam_{\z/n}(\{0,l\}\cup\{p:\pi^{-1}(p) \neq \tau^{-1}(p-l)\})\right)$$
$$\leq 12 \min_{l \in \z/n}\sum_{k \in \z/n} d_{\z/n}(\pi(k) - l,\tau(k)) + 6 \min_{l \in \z/n}diam_{\z/n}(\{0,l\}\cup\{p:\pi^{-1}(p) \neq \tau^{-1}(p-l)\})$$
$$\leq 12 \min_{l \in \z/n}\sum_{k \in \z/n} d_{\z/n}(\pi(k) - l,\tau(k)) + 3 n 
\leq (12 + 9) \min_{l \in \z/n}\sum_{k \in \z/n} d_{\z/n}(\pi(k) - l,\tau(k))$$
$$\leq 21 \times 2 \pi ||\Phi_1(\pi) - \Phi_1(\tau)||_1 
\;\leq\; 21 \times 2 \pi (||\Phi_1(\pi) - \Phi_1(\tau)||_1 + ||\Phi_2(\pi) - \Phi_2(\tau)||_1)$$
\\\textbf{CASE 2: $\min_l \sum_{k} d(\pi(k),\tau(k)+l) < n/3$}.
\\For the lower bound
$$||\Phi_1(\pi) - \Phi_1(\tau)||_1 + ||\Phi_2(\pi) - \Phi_2(\tau)||_1
\leq \min_{l \in \z/n}\sum_{k \in \z/n} d_{\z/n}(\pi(k) - l,\tau(k)) + 4 d(\pi,\tau) \leq 7 d(\pi,\tau),$$
while for the upper bound:
$$d(\pi,\tau) \leq 12 \min_{l \in \z/n}\sum_{k \in \z/n} d_{\z/n}(\pi(k) - l,\tau(k)) + 6 \min_{l \in \z/n}diam_{\z/n}(\{0,l\}\cup\{p:\pi^{-1}(p) \neq \tau^{-1}(p-l)\})$$
$$ \leq 12 \times 2\pi ||\Phi_1(\pi) - \Phi_1(\tau)||_1 + 6 \times 8 ||\Phi_2(\pi) - \Phi_2(\tau)||_1 \leq 12 \times 2\pi (||\Phi_1(\pi) - \Phi_1(\tau)||_1 + ||\Phi_2(\pi) - \Phi_2(\tau)||_1)$$
Combining all the constants from the upper and lower bounds:
$$c_1(\Gamma(Sym_n,\{t,c\})) \leq 7 \times  21 \times 2\pi \approx 923.63 < 1000.$$

\section{\textbf{Remarks on embedding groups into Hilbert space}}\label{EuclideanSection}

This section contains results of Naor and are included with his permission.
Given a metric space $(X,d)$ denote by $c_2(X,d)$ the infimal bi-Lipschitz distortion of $X$ into Hilbert space $L_2$.

\begin{conj}[Cornulier-Tessera-Valette \cite{de2008isometric}]\label{AbelMeetEuclidInfinite}
    .\\Let $G$ be a finitely generated group with a word metric $d_{S}$.
    We have:
    \\$c_2(G,d_S) < \infty$
    $\iff$ $G$ has an abelian subgroup of finite index.
\end{conj}

The following is a finite version of this conjecture.
Given a set $X$ and $r \in \n$ we write $X \choose r$ for the collection of all subsets of $X$ of size $r$.

\begin{conj}[Naor]\label{AbelMeetEuclidFinite}
    Let $\{G_n\}_n$ be a sequence of finite groups with $\sup_n rank(G_n) < \infty$.
    The following five conditions are equivalent:
    \\1. For all $r> \sup_n rank(G_n)$, all Cayley graphs of $\{G_n\}_n$ with $r$ generators embed into Hilbert space with uniformly bounded distortion:
    $$\sup_n \sup_{S \in {G_n \choose r}\;:\;<S> = G_n} c_2(G_n,d_{S}) < \infty$$
    2. There exists $r \in \n$ such that  all Cayley graphs of $\{G_n\}_n$ with $r$ generators embed into Hilbert space with uniformly bounded distortion:
    $$\sup_n \sup_{S \in {G_n \choose r}\;:\;<S> = G_n} c_2(G_n,d_{S}) < \infty$$
    3. For all $r> \sup_n rank(G_n)$ there exists a generating set of size $r$, $S_n$, for each $G_n$, such that the Cayley graphs embed into Hilbert space with uniformly bounded distortion:
    $$\sup_n \inf_{S \in {G_n \choose r}\;:\;<S> = G_n} c_2(G_n,d_{S}) < \infty$$
    4. There exists $r \in \n$ and a generating set of size $r$, $S_n$, for each $G_n$, such that the Cayley graphs embed into Hilbert space with uniformly bounded distortion:
    $$\sup_n \inf_{S \in {G_n \choose r}\;:\;<S> = G_n} c_2(G_n,d_{S}) < \infty$$
    5. There exists $r \in \n$ and a rank $\leq r$ abelian subgroup $H_n \leq G_n$ for each $G_n$ such that
    $$\sup_n [G_n:H_n] < \infty.$$
\end{conj}

In this conjecture, of course $(1) \implies (2) \& (3)$, and $(2) \implies (4)$ and $(3) \implies (4)$.
The direction $(5) \implies (1)$ is not immediate, since one has to deal with arbitrary generating sets of abelian groups.
Nonetheless, it is true:

\begin{prop}[Naor]\label{AbelImpliesEuclid}
    In the above conjecture, $(5) \implies (1)$.
    \\That is, for each $r \in \n$ and $K \in \n$ there exists $D>1$ such that
    \\\textbf{If} $G$ is a finite group which has an abelian subgroup $H \leq G$ of index $[G:H] \leq K$, and $S=\{s_1,...,s_r\}$ is a generating set of $G$
    \textbf{then} $c_2(G,d_S) \leq D$.
\end{prop}

Since $L_2$ embeds bi-Lipschitzly into $L_1$, 
Proposition \ref{AbelImpliesEuclid} implies that abelian groups of bounded rank exhibit Always-$L_1$ behavior.
Also, the proof below gives the dependence $D \lesssim_K r \log r$.
It would be interesting to know the optimal dependence on $r$ in Proposition \ref{AbelImpliesEuclid} (see also Conjecture \ref{ConjectureFlatL_1ToriIntoL_1} below).

\begin{proof}
    \textbf{CASE 1: $G$ is abelian}
    \\The word metric of $y \in G$ is given by:
    $$d(y,0) = \min \{\sum_{i=1}^r|k_i|: y = \sum_{i=1}^r k_i s_i,\;\;\;k_1,...,k_r \in \z\}.$$
    Consider $\Lambda := kernel\{k \mapsto \;\; \sum_{i=1}^r k_i s_i\}$
    which is the kernel of a map $\z^r \to G$ and a lattice of rank $\leq r$.
    We write our group as the direct product of cyclic ones $G = \oplus_{j=1}^J \z_{m_j}$ and pick canonical generators $e_1,...,e_J$, where $e_{j} = (\delta_{jj'})_{j'=1}^J$ is a generator for $\z/m_j$.
    We write each of the canonical generators in terms of the generators $s_1,...,s_r$: there exist $(a_i^{(j)})_{i=1}^r \in \z^r$ for each $j \in [J]$ such that
    $$e_j = \sum_{i=1}^r a_i^{(j)}s_i\;\;\;for\;all\;j \in [J].$$
    That way, given any $g = \sum_{j=1}^J g_j e_j\in G$, we have a representation of $g$ as the sum of the generators $s_1,...,s_r$:
    $$g = \sum_{i=1}^r \left(\sum_{j=1}^J g_j a_i^{(j)}\right) s_i.$$
    Now, if we have any element $(k_1,...,k_r) \in \Lambda$, then we get another representation of $g$,
    $$g = \sum_{i=1}^r \left(k_i + \sum_{j=1}^J g_j a_i^{(j)}\right) s_i.$$
    Conversely, for any two representations of $g$
    $$g = \sum_{i=1}^r k_i s_i = \sum_{i=1}^r k_i' s_i \implies (k_i-k_i')_{i=1}^r \in \Lambda.$$
    Writing $(Ag)_i:=\sum_{j=1}^J g_j a_i^{(j)}$ for each $i \in [r]$, we conclude that the word metric is given by:
    $$d_{word}(g,0) = \min_{k \in \Lambda} \sum_{i=1}^r |k_i - \sum_{j=1}^J g_j a_i^{(j)}| = \min_{k \in \Lambda} ||k - Ag||_{l_1^r}.$$
    
    Form the \textbf{flat torus} $\real^r/\Lambda$ which is a Riemannian manifold (with universal cover $\real^r$) and the Riemannian metric is inherited from the standard Riemannian metric on the cover $\real^r$.
    The metric on the flat torus is given by the formula
    $$d_{\real^r/\Lambda}(x + \Lambda,y + \Lambda) := \min_{z \in \Lambda} ||x - y + z||_{l_2^r},$$
    so by the Cauchy-Schwarz $\leq$ and the above observation, we get that for all $g_1,g_2 \in G$,
    $$d_{\real^r/\Lambda}(Ag_1 + \Lambda,Ag_2 + \Lambda) \leq d_{word}(g_1,g_2)\leq \sqrt{r} \;d_{\real^r/\Lambda}(Ag_1 + \Lambda,Ag_2 + \Lambda).$$
    By Theorem 5.8 in Khot-Naor \cite{khot2006nonembeddability},
    any flat torus $(\real^r/\Lambda, d_{\real^r/\Lambda})$ embeds into $L_2$ with distortion $O(r^{3r/2})$. 
    The dependence on $r$ was improved to $r \sqrt{\log r}$ by Haviv-Regev \cite{haviv2013euclidean} and subsequently to $\sqrt{r \log r}$ by Agarwal-Regev-Tang \cite{agarwal2020nearly}.
    Say the map $f: \real^r/\Lambda \to L_2$ achieves this distortion bound.
    Then the map 
    $$G \to L_2: y \mapsto f(Ay + \Lambda)$$
    is an embedding of distortion $O_r(1)$, where the distortion does NOT depend on $G$, nor on the generators $s_1,...,s_r$.
    \\\textbf{CASE 2: Passing to finite extensions}
    \\We will essentially replicate the proof of the Milnor–Schwarz lemma (see \cite{dructu2018geometric} for instance), and observe that all the involved constants can be bounded by some function of $[G:H]$ (and hence are independent of $G$ and of the generators $s_1,...,s_r$).
    
    The Schreier graph $\Gamma(G/H,S)$ is connected and has $[G:H]$-many points, hence its diameter at most $[G:H]$.
    This means that we may pick representatives for each coset $k_1,...,k_J$ (where $J = [G:H]$) such that
    $$|k_i|_{T} \leq [G:H]\;\;\;and\;\;\cup_{j=1}^J k_j H = G.$$
    Consider the following subset of $H$:
    $$T' := H \cap \{k_{j_{1}}^{-1} s_i k_{j_2}:\; j_1,j_2 \in [J]\;\;and\;\;i \in [r]\}.$$
    We claim that $T'$ generates $H$ and, moreover, the word metric of $T'$ is "undistorted" with respect to the word metric of $T$:
    $$|h|_{T'} \leq |h|_{T} \leq (2[G:H] +1) |h|_{T'}\;\;\;for\;all\;h \in H.$$
    Fix any geodesic of length $M \in \n$ from $1$ to $h \in H$ with generators in $T$, 
    $$1 \to g_1 \to g_2 \to ... \to g_M = h$$
    so $g_{m+1}g_{m}^{-1} = s_{i_m}$ for some $i_m \in [r]$.
    We may express via coset representatives $g_m = k_{j_m} h_m$ where $h_m \in H$ and $j_m \in [J]$.
    Now we have for each $m \in [M]$:
    $$h_{m+1} = k_{j_{m+1}}^{-1} s_i k_{j_m} h_m.$$
    This means that $k_{j_{m+1}}^{-1} s_i k_{j_m} \in H$ and moreover the path
    $$1 \to h_1 \to h_2 \to ... \to h_M = h$$
    is a path from $1$ to $h$ in the Cayley graph of $H$ generated by $T'$.
    This shows that $<T'> = H$ and moreover, $|h|_{T'} \leq |h|_{T}$.
    On the other hand, since $|k_{j_{1}}^{-1} x_i k_{j_2}|_T \leq 2[G:H] +1$, we get that:
    $$|h|_{T'} \leq |h|_{T} \leq (2[G:H] +1)|h|_{T'}\;\;\;for\;all\;h \in H.$$
    Finally, we construct the embedding. 
    From CASE 1, we can find $f: H \to L_2$ with distortion $O_r(1)$.
    Let $([J],d_{trivial})$ be the trivial metric on $[J]$, where any two distinct points have distance $1$.
    (This metric embeds isometrically into Hilbert space as the vertices of the standard simplex.)
    Map $F: G \to L_2 \times [J]: k_j h \to (f(h),j)$.
    We check that for each $k_j h, k_{j'}h' \in G$.
    $$\dfrac{1}{2[G:H]+1}d_{H,T'}(h,h') + 1_{\{j=j'\}}\leq d_{G,T}(k_j h, k_{j'}h') \leq [G:H]1_{\{j=j'\}} + d_{H,T'}(h,h')$$
    and hence we get an embedding of $G$ into $L_2$ whose distortion depends only on $r$ and $[G:H]$.
\end{proof}

One of the main results of Khot-Naor \cite{khot2006nonembeddability}, answering a question of W. B. Johnson,
is that there exists a lattice $\Lambda \leq \real^r$ such that $c_1(l_2^r/\Lambda) \gtrsim \sqrt{r}$.
Case 1 in the proof of Proposition \ref{AbelImpliesEuclid} yields also the following strengthening of this result.

\begin{prop}[Naor]\label{PropositionL_1Quotients}
    For every $r \in \n$ there exists a lattice $\Lambda \subset \real^r$ such that
    $c_1(l_1^r / \Lambda) \gtrsim r$, where the metric on $l_1^r / \Lambda$ is given by:
    $d(x + \Lambda, y + \Lambda) := \inf_{k \in \Lambda } ||x-y+k||_1$ for all $x,y \in \real^r$.
    \\By Cauchy-Schwarz, this implies that $c_1(l_2^r / \Lambda) \gtrsim \sqrt{r}$.
\end{prop}
\begin{proof}
    The proof is probabilistic.
    Recall the following special case of the Alon-Roichman theorem \cite{alon1994random}.
    There exists a constant $c>0$ for which
    we sample $r$ iid elements $S = \{s_1,...,s_r\}$ uniformly at random from $G = (\z/2)^{\lfloor cr\rfloor}$.
    Then with probability $>1/2$, the Cayley graph $\Gamma(G,S)$ has Cheeger constant $>1/10$.
    \\Consider the map $\z^r \to G: (k_i)_{i=1}^r \mapsto \sum_{i=1}^n k_i s_i$ and its kernel $\Lambda:= kernel(k \mapsto \sum_{s \in S} k_s s)$ which is a random lattice $\Lambda \leq \real^r$.
    As we saw in Case 1 in the proof of Proposition \ref{AbelImpliesEuclid},
    the shortest path metric on $\Gamma(G,S)$ embeds isometrically into the $l_1$-torus $l_1^r/\Lambda$.
    We show $c_1(\Gamma(G,S))$ is large.
    \\By Propositions 3.4 and 3.5 in Newmann-Rabinovich \cite{newman2009hard} we have 
    $$\dfrac{1}{|G|^2}\sum_{g \in G} \sum_{h \in G} d_{word}(x,y) \geq \dfrac{diam(\Gamma(G,S))}{2}-1 \gtrsim r.$$
    (The first estimate is immediate: consider the smallest $r$ such that $B|(1,r)| > |G|/2$. Then $LHS \geq r-1 \geq diam(\Gamma(G,S))/2-1$. The second estimate follows from counting all possible small geodesics in an abelian group; see \cite{newman2009hard} for details.)
    \\Fix any $D>1$ and $f: \Gamma(G,S) \to L_1$ with $d(g,h) \leq ||f(g)-f(h)||_1 \leq D d(g,h)$ for all $g,h \in G$.
    By the $L_1$-Poincare inequality for the Cheeger constant (i.e. Theorem 4.7 in \cite{ostrovskii2013metric}), we have with probability $> 1/2$ that
    $$r \lesssim \dfrac{1}{|G|^2}\sum_{g \in G} \sum_{h \in G} ||f(g) - f(h)||_1 \leq \dfrac{20}{|E(\Gamma)|} \sum_{(g,h) \in E(\Gamma)} ||f(g) - f(h)||_1 \leq 20 D.$$
    We conclude that 
    $c_1(l_1^r/\Lambda) \geq c_1(\Gamma(G,S)) \gtrsim r$ with probability $>1/2$.
\end{proof}

By the theorem of Agarwal-Regev-Tang \cite{agarwal2020nearly}, for any lattice $\Lambda \subset \real^r$ we have
$$c_1(l_1^r / \Lambda) \leq \sqrt{r} \;c_1(l_2^r / \Lambda)\leq \sqrt{r} \sqrt{r \log r} = r \sqrt{\log r},$$
hence Proposition \ref{PropositionL_1Quotients} is sharp up to logarithmic factors.
\begin{conj}[Naor]\label{ConjectureFlatL_1ToriIntoL_1}
    For any $r\in \n$ and any lattice $\Lambda \subset \real^n$, 
    $c_1(l_1^r / \Lambda) \lesssim r.$
\end{conj}

.\\We return to bi-Lipschitz embeddings of groups into Hilbert space.

\begin{prop}[A Reduction]\label{Reduction}
    .\\Assuming conjecture 1, the direction $(4) \implies (5)$ in conjecture 2 follows, that is:
    \\\textbf{If} $\{G_n\}_n$ are finite groups and $\{S_n\}_n$ generating sets of size $=r$ with $\sup_n c_2(G_n,d_{S_n}) < \infty$,
    \\\textbf{then} each $G_n$ has an abelian subgroup $H_n$ such that $\sup_n [G_n:H_n] < \infty$.
\end{prop}
\begin{proof}
    \textbf{STEP 1: Compactness and convergence to a limit group}:
    \\We recall the definition of \textbf{Grigorchuk's space of marked groups} $\mathcal{M}_r$ (see Section 6 in \cite{grigorchuk1985degrees}).
    The points of $\mathcal{M}_r$ consist of pairs $(G,S)$ where $G$ is a group, and $S$ is an ordered generating set of size $=r$.
    Every such pair $(G,S)$ is in one-to-one correspondence with the kernel $\ker(\pi)$ of the unique canonical epimorphism from the free group to $G$, $\pi:\mathbb{F}_r \twoheadrightarrow G$, which preserves the order of the generators.
    Each subgroup can be viewed as a subset of $\mathbb{F}_r$, i.e. $\ker(\pi) \in \{0,1\}^{\mathbb{F}_r}$.
    \\The topology of $\mathcal{M}_r$ is the subspace topology inherited from the Tychonoff topology on $\{0,1\}^{\mathbb{F}_r}$.
    One checks that $\mathcal{M}_r$ is a compact topological space (see Proposition 6.1 in \cite{grigorchuk1985degrees}) and that this topology is metrizable (see end of page 286 in \cite{grigorchuk1985degrees}).
    \\For each $n$, order the generators $S_n = \{s_1,...,s_r\}$ in any manner. Consider the sequence of points in the space of marked groups $\{(G_n,S_n)\}_n \subset \mathcal{M}_r$ and pick a convergent subsequence $\{(G_{n_k},S_{n_k})\}_k$ which we relabel and denote by $\{(G_n,S_n)\}_n$ for the shake of notational convenience.
    \\\textbf{STEP 2: The limit group must be virtually abelian}
    \\For every $R \in \n$, there exists $n_0$ such that for all $n>n_0$, $(G_n,S_n)$ and $(G,S)$ satisfy the same relations of length $\leq R$.
    This implies that the ball of radius $R/2$ in the Cayley graph $\Gamma(G_n,S_n)$, is graph-isomorphic to the ball of radius $R/2$ in $\Gamma(G,S)$ and consequently, the word metric on the ball of radius $R/4$ in $\Gamma(G_n,S_n)$ is isometric to that of $\Gamma(G,S)$.
    By the hypothesis, we conclude that every finite subset of $\Gamma(G,S)$ embeds into Hilbert space with distortion $O(1)$.
    \\A theorem of Ostrovskii \cite{ostrovskii2012embeddability} states that for every Banach space $X$ and every locally finite metric space $(M,d)$ (i.e. a space where each ball has finitely many points), if every finite subset of $M$ embeds into $X$ with bi-Lipschitz distortion $\leq D$ for some $D \in (0,\infty)$, then the entire metric space $(M,d)$ embeds with bi-Lipschitz distortion $\leq CD$ where $C\in (0,\infty)$ is a universal constant.
    It follows that the Cayley graph $\Gamma(G,S)$ embeds bi-Lipschiztly into Hilbert space.
    \\A positive answer to Conjecture \ref{AbelMeetEuclidInfinite} would imply that $G$ has an abelian subgroup $H\leq G$ of finite index which we denote by $K := [G:H] < \infty$.
    \\\textbf{STEP 3: The finite groups must also be virtually abelian.} 
    \\This step of the proof was shown to us by Emmanuel Breuillard.
    \\Consider the canonical projection homomorphism $\pi: \mathbb{F}_r \twoheadrightarrow G$.
    Recall that $G \cong \mathbb{F}_r/\ker(\pi)$.
    Consider the subgroup $F := \pi^{-1}(H)$ which has index 
    $$[\mathbb{F}_r:F] = [\mathbb{F}_r/\ker \pi:F/\ker\pi] = [G:H] = K.$$
    By the Milnor–Schwarz lemma (see \cite{dructu2018geometric}) every finite index subgroup of a finitely generated group is finitely generated.
    Pick a finite generating set $F = <w_1,...,w_{r'}>$ (where $r' \in \n$ and $\{w_1,...,w_{r'}\}$ are words in $S \cup S^{-1}$).
    \\Each $G_n$, has its canonical projection map $\pi_n: \mathbb{F}_r \twoheadrightarrow G_n$.
    Define the subgroup: $H_n := \pi_n(F)$.
    \\We have $H_n \cong F/(F \cap \ker \pi_n) \cong (F \ker \pi_n)/ \ker \pi_n$ and the index bound:
    $$[G_n : H_n] = \left[\dfrac{\mathbb{F}_r}{\ker(\pi_n)} : \dfrac{F \ker(\pi_n)}{\ker(\pi_n)}\right] = [\mathbb{F}_r: F \ker(\pi_n)] \leq [\mathbb{F}_r: F] = K.$$
    Since $H$ is abelian, $\pi([w_i,w_j]) = 1$ for all  $i,j \in [r']$.
    There exists $R\in \n$ such that each word $[w_i,w_j]$ ($i,j \in [r']$) has length $< R$.
    By the convergence $(G_n,S_n) \to (G,S)$, there exists $n_0$ such that for all $n > n_0$,
    $(G_n,S_n)$ and $(G,S)$ satisfy the same relation of length $\leq R$,
    and hence
    $$[\pi_n(w_i), \pi_n(w_j)] = \pi_n([w_i,w_j]) = 1 \;\;\;for\;all\;i,j \in [r'].$$
    Therefore subgroup $H_n \leq G_n$ is abelian for all $n>n_0$.
\end{proof}

Combining Proposition \ref{AbelImpliesEuclid} and Proposition \ref{Reduction}, we get:

\begin{coro}
    The conjecture of Cornulier-Tessera-Valette (Conjecture \ref{AbelMeetEuclidInfinite}) \\implies the finite version of Naor (Conjecture \ref{AbelMeetEuclidFinite}).
\end{coro}

It is natural to ask the following $L_2$-variant of Ostrovkii's Question \ref{OstrovskiiQuestion}:
Does there exist $r \in \n$ and generating sets $T_n \subset Sym_n$ each of size $|T_n|\leq r$ such that 
$\sup_n c_2(\Gamma(Sym_n,d_{T_n})) < \infty$?
Assuming Conjecture \ref{AbelMeetEuclidFinite} the answer is no.
The same negative answer follows from the following conjecture:

\begin{conj}[Naor]\label{conjectureForSymmetricGroups}
    For any $r \in \n$ and $M>0$ there exists $n \in \n$ such that 
    for any generating set $T_n$ of $Sym_n$ of size $|T_n|\leq r$ we have the following super-diffusive drift estimate for the simple random walk $\{W_t\}_{t=1}^\infty$ on $\Gamma(Sym_n,T_n)$ starting at $W_0 = 1$:
    $$\sup_{t \in \n} \dfrac{\e \;d_{T_n}(1,W_t)^2}{t} > M.$$
\end{conj}

Potentially, stronger super-diffusivity estimates might hold.
Using Markov type $2$ as in Linial-Magen-Naor \cite{linial2002girth} (alternatively see Chapter 8 in \cite{ostrovskii2013metric} or Chapter 13 in \cite{lyons2017probability})
Conjecture \ref{conjectureForSymmetricGroups} implies that for any $r \in \n$ and any sequence of generating sets $\{T_n \subset Sym_n\}$ of size $|T_n|\leq r$ we have
$\sup_n c_2(Sym_n,d_{T_n}) = \infty$.

Next, we verify Conjecture \ref{conjectureForSymmetricGroups} for specific generating sets of the symmetric groups.

Let $T_n = \{(0123...n-1),(01)\}$ is the cycle and the transposition.
Then one checks that the ball of radius $n/4$ of $Sym_n$ is graph-isomorphic to the ball of radius $n/4$ of the $1D$-lampshuffler group $Sym_{oo}(\z) \ltimes \z$ with standard generators \cite{yadin2009rate}. 
(In other words, we have local convergence to the $1D$-lampshuffler group.)
By a theorem of Yadin \cite{yadin2009rate}, the simple random walk on $Sym_{oo}(\z) \ltimes \z$ has \textit{drift exponent} $3/4$.
This means that for small times the same estimate holds on the symmetric group:
$$\e d(1,W_t) \gtrsim t^{3/4}\;\;\;for\;all\;1 \leq t \leq n/4.$$
This shows Conjecture \ref{conjectureForSymmetricGroups} for $T_n = \{(0123...n-1), (01)\}$.

In what follows, we need to recall the following \textbf{Basic Fact}:
The random walk $\{W_t\}_{t=1}^\infty$ on an $n$-vertex degree $d$ graph $\Gamma$ with spectral gap $1-\lambda$ satisfies:
$$\e d(W_0,W_t) \gtrsim_{d,\lambda} t\;\;\;for\;all\;0\leq t \lesssim_{d,\lambda} \log n,$$
where $d(\cdot,\cdot)$ is the shortest path metric on $\Gamma$, and we have suppressed the dependence on the degree $d$ and the spectral gap $1-\lambda$.
(This estimate goes back to Theorem 5 in Kesten \cite{kesten1959symmetric}; see also Proposition 6.9 in \cite{lyons2017probability}.).

Let $T_n$ be the generating set from Kassabov's theorem.
The Cayley graphs are expanders, so Conjecture \ref{conjectureForSymmetricGroups} is verified by the above fact.

Finally, we show that uniformly random generators satisfy Conjecture \ref{conjectureForSymmetricGroups}.
Recall the theorem of Dixon \cite{dixon1969probability}: an iid sample of uniformly random permutations $\pi_1,... \pi_r$ ($r\geq 2$) generate the subgroup $Alt_n$ (the alternating group) or the entire symmetric group $Sym_n$ with probability $> 1 - O(1/\log \log n) = 1- o(1)$.
In particular, the probability that the generated subgroup is $Alt_n$ equals $1/2^r + o(1)$.

\begin{prop}[Naor]
    For each $r\geq 2$ let $\pi_1,...,\pi_r \in Sym_n$ be iid sampled permutations from the uniform distribution. Then with high probability 
    $$\e d_{\{\pi_1,...,\pi_r\}}(1,W_t) \gtrsim t\;\;\;for\;all\;\;t=1,...,O(\log n),$$
    where $\{W_t\}_{t=1}^\infty$ is the simple random walk on $<\pi_1,...,\pi_r>$.
\end{prop}
\begin{proof}[First proof]
    Look at the action of the group $<\pi_1,...,\pi_r> \curvearrowright  \binom{[n]}{4}$ on the collection of all $4$-element subsets on $n$.
    Build the Schreier graph $\Gamma_n(4)$ with vertices $\binom{[n]}{4}$, where we connect $(A, \pi_i[A])$ for each generator $\pi_i$ and subset $A \in \binom{[n]}{4}$.
    \\By a theorem of Friedman-Joux-Roichman-Stern-Tillich \cite{friedman1998action}, 
    $\Gamma_n(4)$ is an expander with high probability.
    Denote by $W_t$ the simple random walk on $Sym_n$ (or $Alt_n$)
    and by $W_t(\{1,2,3,4\})$ the image of the set $\{1,2,3,4\}$ under the permutation $W_t$, and by $d_{\Gamma_n(4)}$ the shortest path metric on the expander graph.
    By the basic fact above we get:
    $$\e d_{\Gamma_n(4)}(\{1,2,3,4\},W_t(\{1,2,3,4\})) \gtrsim t\;\;\;for\;all\;\;t=1,...,O(\log n).$$
    \\The shortest path metric on the Schreier graph is at most by the shortest path metric on the Cayley graph (as the latter covers the former).
    This means that
    $$\e \;d_{\{\pi_1,...,\pi_r\}}(1,W_t) \gtrsim t\;\;\;for\;all\;\;t=1,...,O(\log n).$$
\end{proof}

\begin{proof}[Second proof]
    By a theorem of Dixon-Pyber-Seress-Shalev\cite{dixon2003residual} for any nonempty freely reduced word $w$ in the generators and inverses $\{\pi_1,...,\pi_r, \pi_1^{-1},..., \pi_r^{-1}\}$, 
    $p_n := \p[w = 1] \to 0$ as $n \to \infty$.
    By the union bound on all words of length $< \log_{2r} 1/p_n - 1$, we conclude that the
    $$\e \;girth(\Gamma(Sym_n, \{\pi_1,...,\pi_r\})) \gtrsim \log \dfrac{1}{p_n} \to \infty\;\;\;as\;n \to \infty.$$
    For small times, the simple random walk on a degree $2r$ regular graph of high girth is identical to that of a $2r$-regular tree, therefore $\e d_{\{\pi_1,...,\pi_r\}}(1,W_t) \gtrsim t$ for times $1\leq t \leq girth(\Gamma)/2$.
\end{proof}


\section{\textbf{Proof of test space characterization of type}}\label{SectionProofOfCorollaries}

As was already discussed in subsection \ref{SubSectionRibe},
if $X$ has trivial Rademacher type, then $L_1$ is finitely representable into $X$.
By Theorem \ref{MainTheorem}, for each $n \in \n$, $\Gamma(Sym_n,\{c,t\})$ embeds into $l_1^m$ with distortion $<1000$ for sufficiently high $m$,
hence $\Gamma(Sym_n,\{c,t\})$ embeds into $X$ with distortion $<1000$.
For the other direction, it suffices to show the following lemma.
The embedding is similar to that of Arzhantseva, Guba and Sapir \cite{arzhantseva2006metrics} for the lamplighter $\z/n \wr \z/n$.

\begin{lemm}
    For each $n \in \n$, the Hamming cube $(\{0,1\}^n,||\cdot||_1)$ 
    \\embeds into $\Gamma(Sym_{4n^2},\{c,t\})$ with bi-Lipschitz distortion $\leq 30$.
\end{lemm}
\begin{proof}
For each $n\in \n$ consider the following $n$ disjoint transpositions in $Sym_{4n^2}$
$$t_1:=(0,n), t_2:=(1\;n+1), t_3:=(2\;n+2), ... , t_n:=(n-1\;2n-1).$$
Map each $(\epsilon_1,...,\epsilon_n) \in \{0,1\}^n$ to the permutation
$f_\epsilon := t_1^{\epsilon_1}...t_n^{\epsilon_n}$.
\\Fix distinct $\epsilon = (\epsilon_1,...,\epsilon_n), \delta = (\delta_1,...,\delta_n) \in \{0,1\}^n$.
By the word metric formula in Section \ref{SectionYadinFormula}, we have the upper bound:
$$d(f_\epsilon,f_\delta) 
\leq \min_{l \in \z/4n^2}\left(6\sum_{k \in \z/4n^2} d_{\z/4n^2}(f_\epsilon(k) - l,f_\delta(k)) + 2diam_{\z/n^2}(\{0,l\}\cup\{p:f_\epsilon^{-1}(p) \neq f_{\delta}^{-1}(p-l)\})\right) $$
$$\leq 6\sum_{k \in \z/4n^2} d_{\z/4n^2}(f_\epsilon(k),f_\delta(k)) + 4n = 6n ||\epsilon-\delta||_1 + 4n \leq 10n ||\epsilon-\delta||_1,$$
and the lower bound:
$$d(f_\epsilon,f_\delta) 
\geq \dfrac{1}{3} \min_{l \in \z/4n^2}\sum_{k \in \z/4n^2} d_{\z/4n^2}(f_\epsilon(k) - l,f_\delta(k))$$
$$ = \dfrac{1}{3} \min_{l \in \z/4n^2}\left((4n^2-2n) d_{\z/4n^2}(0,l) + \sum_{i:\epsilon_i = \delta_i} d_{\z/4n^2}(0,l) + \sum_{\epsilon_i = 0,\delta_i = 1} d_{\z/4n^2}(n,l) +  \sum_{\epsilon_i = 1,\delta_i = 0} d_{\z/4n^2}(-n,l)\right).$$
By inspection, it is clear that the minimizer occurs at $l=0$ so we get
$$ \dfrac{1}{3}n ||\epsilon - \delta||_1 \leq d(f_\epsilon,f_\delta) \leq 10n ||\epsilon-\delta||_1.$$
\end{proof}


\section*{\textbf{Acknowledgments}}
I thank Assaf Naor for his results which have been included in Section \ref{EuclideanSection} of this paper, for encouragement, insightful discussions, writing advice, tons of context around this work, and for helping me with the use of representation theory for the embedding in Section \ref{SectionFirstTermEmbedding}.
I thank Mikhail Ostrovskii for telling me his Question \ref{OstrovskiiQuestion}, and for insightful discussions and context.
I thank Emmanuel Breuillard for showing me step 3 in Proposition \ref{Reduction}.
I also thank Kunal Chawla, Seung-Yeon Ryoo, and Eduardo Silva for insightful discussions and context.
Finally, I thank the Casa Matemática Oaxaca (CMO) and Banff International Research Station for Mathematical Innovation and Discovery (BIRS), as well as the American Institute of Mathematics (AIM) for giving me the opportunity to speak with Mikhail Ostrovskii and learn about his question.

\newpage
\bibliographystyle{plain}
\bibliography{references}

\end{document}